\documentclass[11pt,a4paper]{article}

\usepackage[ngerman,english]{babel}

\usepackage{hyperref}
\usepackage{latexsym,amsmath,amstext,amsthm,amscd,amsopn,verbatim,amsrefs}

\usepackage{graphicx}
\usepackage{geometry}

\usepackage[dvips]{color}
\usepackage{epsfig}
\usepackage[all]{xy}
\usepackage{amscd}
\usepackage{amssymb}
\usepackage{latexsym}
\usepackage{amsfonts}

\renewcommand{\abstract}{\textbf{Abstract}}

\renewcommand{\d}{{\rm d}}

\newcommand{\I}{\mathrm{i}}
\DeclareMathOperator{\id}{id}

\newtheorem{satz}{Satz}[subsection]

\newtheorem*{theorem*}{Theorem}

\newtheorem{proposition}[satz]{Proposition}

\newtheorem*{definition*}{Definition}

\date{}

\begin{document}
\title{\vspace{-1.2cm}\textbf{\textit{On Moyal-Weyl deformations of $\mathrm{DGA}$ and $\mathrm{DGCA}$}}}
\author{\textit{Johannes L\"offler}}
\maketitle
\begin{abstract} 
We consider a natural variant of the Moyal-Weyl product and show that it yields a functorial deformation of differential graded algebras and that we can deform coalgebras in a similar way. The Moyal-Weyl deformation of graded algebras has already been introduced by Fedosov \cite{Fe} and also appears in \cite{Ge} as part of an $A_\infty$ structure, but we are not aware that the analogous result for differential graded coalgebras already appeared in the literature and discuss the Moyal-Weyl  of the category of complexes as an example. 
\end{abstract}
\subsection*{\textit{Introduction}}
In the last decades deformation theories appeared in many branches of mathematics and physics. The original idea of deformation seems quite hard to trace back, but some of the original literature surely are \cite{Dir}, \cite{BFFLS1}, \cite{BeGCoQu}, \cite{Moy} and finally the impressive proof of deformation quantization \cite{K}. 

In quantization procedures it is common to denote the formal deformation parameter by $\hbar$ in honour of Max Planck. Of course in physics
$\hbar\approx6.58211928...\times10^{-16}eV\cdot{s}$
is not really a variable, neither a constant that we are able to determine exactly in experiments, therefore we expect quantizations to be rigid with respect to fluctuations of $\hbar$. Notice that the dimension of $\hbar$ is
$[mass][length]^2/[time]$
and inverse to the dimension of the ordinary Poisson bracket on the phase space $T(\mathbb{R}^{d})$ we will soon describe in formula \ref{CanPoi}. %or the theory should determine the exact value, but this seems to be far from reality. 

In the deformation quantization of a Poisson manifold $(M,\Pi)$ the definition of a $\star$ product is central: A $\star$ product is a $\mathbb{C}[\![\hbar]\!]$-bilinear associative operation $\star:{C^{\infty}(M)[\![\hbar]\!]}\times{C^{\infty}(M)[\![\hbar]\!]}\rightarrow{C^{\infty}(M)[\![\hbar]\!]}$ of the shape
$f\star{g}=\sum_{n=0}^{\infty}\hbar^{n}B_{n}(f,g)$
with bi-differential operators $B_n$ and with the properties $1\star{f}=f=f\star1$, $B_{0}(f,g)=fg$ and $B_{1}(f,g)-B_{1}(g,f)=\I\lbrace{f,g}\rbrace$ where $\lbrace{\cdot,\cdot}\rbrace$ is the Poisson bracket. Two star products on $(M,\Pi)$ are equivalent if there exists a formal series
$S=id+\sum_{l=1}^{\infty}\hbar^{l}S_{l}$
of $\mathbb{C}[[\hbar]]$-linear operators $S_{l}:{C^{\infty}(M)[[\hbar]]}\rightarrow{C^{\infty}(M)[[\hbar]]}$ with $S_{l}(1)=0$ for $l\geq1$ and
\begin{equation}\label{Equi}
f\star{g}=S^{-1}(Sf\star'Sg)\;\forall\;f,g\in{C^{\infty}(M)[[\hbar]]}
\end{equation}

For the ordinary Poisson bracket on the phase space $T(\mathbb{R}^{d})$ given by
\begin{equation}\label{CanPoi}
\{f,g\}=\sum_{k=1}^{d}\Biggr(\frac{\partial{f}}{\partial{q^k}}\frac{\partial{g}}{\partial{p_k}}-\frac{\partial{f}}{\partial{p_k}}\frac{\partial{g}}{\partial{q^k}}\Biggr)
\end{equation}
the standard $\star$ product, named Weyl-Moyal product, is the binary operation defined by
$$f\star{g}:=\mu\circ\exp\Bigr(-\I\hbar\frac{\partial}{\partial{p_k}}\otimes\frac{\partial}{\partial{q^k}}\Bigr)f\otimes{g}$$
where $\mu$ is the multiplication \cite{Groe}. The general Moyal-Weyl product formula for $(\mathcal{A},\mu)$ an associative algebra and $D^i,D_i:\mathcal{A}\rightarrow\mathcal{A}$ commuting derivations is
\begin{equation}\label{WM}\mu\circ\exp\left(-\I\hbar\sum_{i=1}^{n}D^i\otimes{D}_i\right)\end{equation}
where we use the Koszul sign conventions
$(f\otimes{g})(v\otimes{w})=(-1)^{\vert{g}\vert\vert{v}\vert}f(v)\otimes{g(w)}$ and $(f\otimes{g})\circ(f'\otimes{g'})=(-1)^{\vert{g}\vert\vert{f'}\vert}f\circ{f'}\otimes{g}\circ{g'}$. Formula \ref{WM} defines a deformed associative product \cite{Gu}, for a detailed proof of the associativity of \ref{WM} we refer to \cite{WaB}. This product is an important ingredient in the Fedosov deformation quantization of symplectic manifolds, see \cite{J} for a construction of Fedosov $\star$ products on K\"ahler manifolds of constant sectional curvature.

\textbf{Acknowledgements:} I thank H.-Y. Yeh for discussions and pointing out that it could be interesting to consider the bosonic-fermionic aspects of the deformation. I thank C.-A. Abad, D. Roytenberg and G. Schaumann for discussions, reading of draft versions. Last but not least I thank the MPIM in Bonn. johannes@mpim-bonn.mpg.de%for providing a perfect working environment and financial support.

\section{\textit{Functorial deformation of differential graded algebras}}
This paper is inspired by deformation quantization, but considers a different input data.
The deformation of the exterior algebra is one of many natural examples that appear in the mathematical literature and physics where one can apply the following proposition \ref{De1}. The deformation \ref{De1} is in some sense just the simplest incarnation of the Moyal-Weyl product \ref{WM} where we choose $n=1$ and $D^i=D_i=\d$ where $\d$ squares to zero:
\begin{proposition}\label{Functorialdeformation} Let $(A,\wedge)$ be a graded algebra over $\mathbb{C}$ and $\d$ with $\d^2=0$ be an odd degree super derivation. The following formula defines a deformed associative product
\begin{align}\label{De1}
a_1\wedge^{\d}{a_2}=a_1\wedge{a_2}+\I\hbar(-1)^{\vert{a_1}\vert}\d{a_1}\wedge\d{a_2}
\end{align}
\end{proposition}
\begin{proof} The proof only uses that $\d$ squares to zero and is an odd degree super derivation of the graded product $\wedge$, for instance we assume the graded Leibniz compatibility $\d(a_1\wedge{a_2})=\d{a_1}\wedge{a_2}+(-1)^{\vert{a_1}\vert}a_1\wedge\d{a_2}$.
The detailed proof of the associativity goes as follows
$$a\wedge^{\d}(b\wedge^{\d}{c})=a\wedge^{\d}\left[b\wedge{c}+\I{\hbar}(-1)^{\vert{b}\vert}\d{b}\wedge\d{c}\right]$$
$$=a\wedge{b}\wedge{c}+\I{\hbar}\left[(-1)^{\vert{a}\vert}\d{a}\wedge\d(b\wedge{c})+(-1)^{\vert{b}\vert}a\wedge\d{b}\wedge\d{c}\right]$$
\begin{align}\label{ASS}=a\hspace{-0.07cm}\wedge\hspace{-0.07cm}{b}\hspace{-0.07cm}\wedge\hspace{-0.07cm}{c}\hspace{-0.05cm}+\hspace{-0.05cm}\frac{(\hspace{-0.05cm}-1\hspace{-0.05cm})^{\vert{a}\vert}\d{a}\hspace{-0.07cm}\wedge\hspace{-0.07cm}\d{b}\hspace{-0.07cm}\wedge\hspace{-0.07cm}{c}+(\hspace{-0.05cm}-1\hspace{-0.05cm})^{\vert{a}\vert+\vert{b}\vert}\d{a}\hspace{-0.07cm}\wedge\hspace{-0.07cm}{b}\hspace{-0.07cm}\wedge\hspace{-0.07cm}\d{c}+(\hspace{-0.05cm}-1\hspace{-0.05cm})^{\vert{b}\vert}a\hspace{-0.07cm}\wedge\hspace{-0.07cm}\d{b}\hspace{-0.07cm}\wedge\hspace{-0.07cm}\d{c}}{1/{\I\hbar}}\end{align}
$$=a\wedge{b}\wedge{c}+\I{\hbar}\left[(-1)^{\vert{a}\vert}\d{a}\wedge\d{b}\wedge{c}+(-1)^{\vert{a}\wedge{b}\vert}\d(a\wedge{b})\wedge\d{c}\right]$$
$$=\left[a\wedge{b}+\I{\hbar}(-1)^{\vert{a}\vert}\d{a}\wedge\d{b}\right]\wedge^{\d}{c}=(a\wedge^{\d}{b})\wedge^{\d}{c}$$\end{proof}
\subsubsection*{\textit{Remarks}}
\begin{itemize}
\item With the Koszul sign convention we could also rewrite the deformation \ref{De1} in the shape
$\wedge^d=\wedge+\I\hbar\wedge(\d\otimes\d)$ and analog the formula
$a_1\wedge^{\d}{a_2}=a_1\wedge{a_2}+\I\hbar(-1)^{\vert{a_1}\vert\vert\d\vert}\d{a_1}\wedge\d{a_2}$
is a deformation of any graded associative algebra if $\d$ is a derivation with zero square, but notice if $\d$ is of even degree, {\em i.e.} $\d(a\cdot{b})=\d{a}\cdot{b}+a\cdot\d{b}$ the vanishing $\d^2=0$ would also imply $\d^2({a}\cdot{b})=2\d{a}\cdot\d{b}=0$, in other words the analog deformation with formula \ref{De1} does not make non-trivial sense if $\d$ is even.

\item Notice that the added deformation term $\I\hbar(-1)^{\vert{a_1}\vert}\d{a_1}\wedge\d{a_2}$ only depends on $a_1$ and $a_2$ modulo closed terms. We can also write for example
$a_1\wedge^{\d}{a_2}=a_1\wedge{a_2}+\I\hbar(-1)^{\vert{a_1}\vert}\d\left({a_1}\wedge\d{a_2}\right)$ or the more symmetric version
$a_1\wedge^{\d}{a_2}=a_1\wedge{a_2}+\frac{\I\hbar}{2}\d\left(-\d{a_1}\wedge{a_2}+(-1)^{\vert{a_1}\vert}{a_1}\wedge\d{a_2}\right)$.
The meaning of this rewriting is that the added deformation term is exact.

Let us at this point recall Stokes theorem
$\int_{M}\d\alpha=\int_{\partial{M}}\alpha$
for $\alpha\in\Omega^{\dim(M)-1}(M)$ a differential form on a manifold $M$, hence integration of exact forms of maximal degree is trivial if the boundary of $M$ is empty.
\item Clearly the product $\wedge^\d$ is not graded in the sense of the $\wedge$ grading, but it is a graded product if we as usual give the formal parameter $\hbar$ degree $-2$. Notice that this quantization is a first order deformation, hence it has no convergence problems or issues with physical dimensions of the deformation parameter $\hbar$.
\item As for the wedge product the deformation $\wedge^\d$ still inherits the property that a $\wedge^\d$ product of two bosons $\in\Omega^{2\bullet}(M)[[\hbar]]$ as well as the $\wedge^\d$ product of two fermions $\in\Omega^{2\bullet+1}(M)[[\hbar]]$ are ``bosonic" sums and the $\wedge^\d$ product of a boson with a fermion is a ``fermionic" sum, here we have copied the boson-fermion language from Witten's article on super-symmetry and Morse theory \cite{WiSM}.

In the following we will assume that the deformation parameter $\hbar$ and differential $\d$ are real, {\em i.e.} $\overline{\hbar}=\hbar$ and $\overline{\d{a}}=\d\overline{a}$. The ``{Pauli exclusion principle}"
$a_1\wedge{a}_2=(-1)^{\vert{a_1}\vert\vert{a_2}\vert}a_2\wedge{a}_1$
does no longer hold for $\wedge_\d$ and we just have the ``{weak Pauli exclusion principle}"
$\overline{a_1\wedge^{\d}{a}_2}=(-1)^{\vert{a_1}\vert\vert{a_2}\vert}\overline{a}_2\wedge^{\d}\overline{a}_1$.
Two fermions in this model can {\em a priori} non trivially be in the same state $f$ and create the exact boson
$f\wedge^\d{f}=(\hbar/\I)\d\left({f}\wedge\d{f}\right)$.
Beside this difference also for the deformed product the formulas
$\overline{b\wedge^{\d}{f}}=\overline{f}\wedge^{\d}\overline{b}$,
$\overline{f_1\wedge^{\d}{f_2}}=-\overline{f_2}\wedge^{\d}\overline{f_1}$
and $\overline{b_1\wedge^{\d}{b_2}}=\overline{b_2}\wedge^{\d}\overline{b_1}$ hold, {\em i.e.} the bosons $(\Omega^{2\bullet}(M)[[\hbar]],\wedge^\d)$ form a *-algebra over $\mathbb{C}$.

For the bosons there is another proof of the associativity of \ref{De1}: With help of the invertible maps
$S:=\id+\sqrt{\hbar/\I}\d$ and ${S}^{-1}=\id-\sqrt{\hbar/\I}\d$:
it is easy to verify that we have an associative product by the equivalence formula
$$S^{-1}\left(Sa_1\wedge{S}a_2\right)=a_1\wedge{a_2}+\sqrt{\hbar/\I}\Big(1-(-1)^{\vert{a_1}\vert}\Bigr)a_1\wedge\d{a_2}+\I\hbar(-1)^{\vert{a_1}\vert}\d{a_1}\wedge\d{a_2}$$
\item As usual introducing the anti-commuting coordinates $\mathrm{V}^i=\d{x}^i$ we can interpret the term $\d{a}\wedge\d{b}=\mathrm{V}^i\mathrm{V}^j\partial_i{a}\wedge\partial_j{b}$ as a super-Poisson bracket, we thank D. Roytenberg for pointing out this interpretation.
\end{itemize}

It is easy to realize that $\d$ is also a super derivation with respect to $\wedge^\d$ and if a $\mathbb{C}$-linear morphism
$\varphi:(A,\wedge_A,\d_A)[[\hbar]]\rightarrow(B,\wedge_B,\d_B)[[\hbar]]$
enjoys the compatibilities $\varphi\d_A=\d_B\varphi$ and $\varphi(a_1\wedge_A{a_2})=(\varphi{a}_1\wedge_B{\varphi{a_2}})$
the same identities remain true if we replace $\wedge$ by $\wedge^\d$, hence some of the previous observations can be summarized in a functorial statement:

Consider the category $\mathrm{DGA}$ with objects differential graded algebras $(A,\wedge,\d)[[\hbar]]$ over $\mathbb{C}$ and with arrows structure compatible maps
$\varphi:(A,\wedge_A,\d_A)[[\hbar]]\rightarrow(B,\wedge_B,\d_B)[[\hbar]]$
{\em i.e.} degree respecting, compatible $\mathbb{C}$-linear morphisms. 
\begin{proposition}\label{Fu1}
By $\mathcal{D}(A,\wedge_A,\d_A)[[\hbar]]=(A,\wedge_A^{d_A},\d_A)[[\hbar]]$ and $\mathcal{D}(\varphi)=\varphi$ we have defined a faithful functor $\mathcal{D}:\mathrm{DGA}[[\hbar]]\hookrightarrow{\mathrm{DGA}}[[\hbar]]$. We can also consider \ref{Functorialdeformation} as a faithful functor $\mathcal{D}:\mathrm{DGA}\hookrightarrow{\mathrm{DGA}}[[\hbar]]$.
\end{proposition}
\subsection{\textit{Example: Moyal-Weyl deformation of the category of complexes}}
We discuss an application of \ref{Functorialdeformation}: From the deformation quantization point of view the most interesting $\mathrm{DGA}$s are
the space of polyvector fields $T_{poly}(M)=\oplus_{k=-1}^{\infty}{\Gamma}^{\infty}({\wedge}^{k+1}TM)$ and the vector space
$D_{poly}(M)=\oplus_{k=-1}^{\infty}\lbrace{\text{Polydifferential operators}:{C}^{\infty}(M)^{\otimes{k+1}}\rightarrow{C}^{\infty}(M)\rbrace}$
of polydifferential operators on a manifold $M$: We denote by $[\cdot,\cdot]_{S-N}$ the Schouten-Nijenhuis bracket on $T_{poly}(M)$ and a Poisson structure $\Pi$ induces a flat derivation
$\d_\Pi:=[\Pi,\cdot]_{\mathrm{S-N}}:T_{poly}^\bullet(M)\rightarrow{T}_{poly}^{\bullet+1}(M)$
of the $\wedge$ product. The famous Hochschild differential
$\d_H:D_{poly}^\bullet(M)\rightarrow{D}_{poly}^{\bullet+1}(M)$
is a graded flat derivation of the associative cup product $\cup$. In other words the input data of a Poisson tensor endows $T_{poly}(M)$ with a $\mathrm{DGA}$ structure and $D_{poly}(M)$ carries a natural $\mathrm{DGA}$ structure. In both cases the direct application of \ref{Functorialdeformation} looks somehow strange, for example for $(T_{poly}(M),\d_\Pi)$ the in general not invertible Poisson structure $\Pi$ appears ``{quadratic}" in the first order of $\hbar$. However, nevertheless let us mention the following example how to cure this strange ``{quadratic}" deformation into the direction that $\Pi$ can appear in some sense ``{linear}":

Consider the category with objects chain complexes $(C^\bullet,\d^\bullet)$ over $\mathbb{C}$, {\em i.e.} graded vector spaces $C^\bullet$ equipped with linear, square zero maps $\d:C^\bullet\rightarrow{C}^{\bullet+1}$ and the arrows are linear maps $\varphi:C_k^\bullet\rightarrow{C}_l^{\bullet+\vert\varphi\vert}$ where we do not assume that this maps intertwine the boundary maps $\d^\bullet_k,\d^\bullet_l$. It is well-known that there is a differential of degree $1$ acting on the space of graded linear maps
$\varphi:C_k^\bullet\rightarrow{C}_l^{\bullet+\vert{\varphi}\vert}$
by the graded homotopy formula
$\d{\varphi}:=\d_l\circ\varphi-(-1)^{\vert\varphi\vert}\varphi\circ\d_k$.
This natural differential defines a graded derivation of the composition of linear maps, {\em i.e.} whenever it makes sense we have a graded, associative ``product" defined by composition $\varphi\circ\alpha$ for $\alpha:C_j^\bullet\rightarrow{C}_k^{\bullet+\vert\alpha\vert}$ and $\varphi:C_k^\bullet\rightarrow{C}_l^{\bullet+\vert\varphi\vert}$ and also the graded Leibniz rule
$\d(\varphi\circ\alpha)=\d\varphi\circ\alpha+(-1)^{\vert\varphi\vert}\varphi\circ\d\alpha$
holds. Hence analog to \ref{Functorialdeformation} we can deform this composition by the formula
\begin{align*}
\varphi\circ^{\d}\alpha&=\varphi\circ\alpha+\I\hbar\bigr((-1)^{\vert\varphi\vert}\d_l\circ\varphi-\varphi\circ\d_k\bigr)\circ\bigr(\d_k\circ\alpha-(-1)^{\vert\alpha\vert}\alpha\circ\d_j\bigr)\end{align*}
\begin{align*}=\varphi\circ\alpha+\I\hbar\Bigr(&\hspace{-0.07cm}(\hspace{-0.03cm}-1\hspace{-0.02cm})^{\vert\varphi\vert}\d_l\circ\varphi\circ\d_k\circ\alpha\hspace{-0.07cm}-\hspace{-0.07cm}(\hspace{-0.03cm}-1\hspace{-0.02cm})^{\vert\alpha\vert+\vert\varphi\vert}\d_l\circ\varphi\circ\alpha\circ\d_j\hspace{-0.07cm}+\hspace{-0.07cm}(\hspace{-0.03cm}-1\hspace{-0.02cm})^{\vert\alpha\vert}\varphi\circ\d_k\circ\alpha\circ\d_j\Bigr)
\end{align*}
This deformed composition still respects the identity morphisms and differentials in the sense $\varphi\circ^{\d}{\id}=\varphi=\id\circ^{\d}\varphi$ and $\d_l\circ^{\d}\varphi=\d_l\circ\varphi$ and $\varphi\circ^{\d}{\d_k}=\varphi\circ{\d_k}$.

\section{\textit{``{Defect}" for differential graded Lie algebras}}
Here we consider the analog of \label{Functorialdeformation} and  \ref{Functorialcodeformation} for $\mathrm{DGLA}$s, {\em i.e.} we assume a graded $\mathbb{C}$-bilinear bracket $[\cdot,\cdot]$ that satisfies the graded Jacobi identity
$$(-1)^{\vert{a_1}\vert\vert{a_3}\vert}[a_1,[a_2,a_3]]+(-1)^{\vert{a_2}\vert\vert{a_1}\vert}[a_2,[a_3,a_1]]+(-1)^{\vert{a_3}\vert\vert{a_2}\vert}[a_3,[a_1,a_2]]=0$$
and a graded $\mathbb{C}$-linear map of degree $1$ that satisfies the graded Leibniz rule $\d[a_1,a_2]=[\d{a_1},a_2]+(-1)^{\vert{a_1}\vert}[a_1,\d{a_2}]$.
\begin{proposition}
The operation $[\cdot,\cdot]_{\d}$ defined by
$$[a_1,{a_2}]^{\d}:=[a_1,{a_2}]+\I\hbar(-1)^{\vert{a_1}\vert}[\d{a_1},\d{a_2}]$$
has an exact defect of graded Jacobi identity.\end{proposition}
\begin{proof}
By calculation
\begin{align*}\frac{2\I\hbar}{3}\biggr(&(-1)^{\vert{a_1}\vert\vert{a_3}\vert+\vert{a_1}\vert+\vert{a_2}\vert}[a_1,[a_2,\d{a_3}]]-(-1)^{\vert{a_1}\vert\vert{a_3}\vert}[\d{a_1},[a_2,a_3]]\\&+(-1)^{\vert{a_2}\vert\vert{a_1}\vert+\vert{a_2}\vert+\vert{a_3}\vert}[a_2,[a_3,\d{a_1}]]-(-1)^{\vert{a_2}\vert\vert{a_1}\vert}[\d{a_2},[a_3,{a_1}]]\\&+(-1)^{\vert{a_3}\vert\vert{a_2}\vert+\vert{a_1}\vert+\vert{a_2}\vert}[a_3,[a_1,\d{a_2}]]-(-1)^{\vert{a_3}\vert\vert{a_2}\vert}[\d{a_3},[a_1,{a_2}]]\biggr)\end{align*}
is an explicit primitive of the Jacobiator, hence the cohomology class of the defect vanishes.
\end{proof}
\subsubsection*{\textit{Remark}}
\begin{itemize}
\item The question arises if it is possible to construct for $n\geq4$ Taylor components $Q_n:S^n(\mathfrak{g}[1])\rightarrow\mathfrak{g}[1]$ to upgrade this defect deformation to a functor from $\mathrm{DGLA}$ to the category of $L_\infty$ algebras. The previous question has already been considered in the literature, we thank D. Roytenberg for pointing this out: The version of higher derived brackets of Getzler \cite{Get} is a convenient positive answer and  the higher derived brackets of Voronov \cite{Vor} seem related, let us mention the article \cite{Cat}. %but notice there the situation differs because of the presence of a projector $\mathrm{P}$ with $\mathrm{P}^2=\mathrm{P}$ and $\mathrm{P}[a,b]=\mathrm{P}[\mathrm{P}a,b]+\mathrm{P}[a,\mathrm{P}b]$ and this second equation is obviously not always true for $\mathrm{P}=\id$, we refrain from going into the details and refer to \cite{Cat}. 
\end{itemize}
\section{\textit{Deformation of graded coalgebras with codifferential}}
\begin{proposition}\label{Functorialcodeformation} Let $\d$ be an odd degree flat coderivation of a graded coproduct $\Delta$ over $\mathbb{C}$. With the Koszul sign convention we have a deformed coproduct 
$$\Delta^\d:=\Delta+\I\hbar(\d\otimes\d)\Delta$$
\end{proposition}
\begin{proof}
This is just dually to \ref{Fu1}: Again the proof of the coassociativity
$(\id\otimes\Delta^\d)\Delta^\d=(\Delta^\d\otimes\id)\Delta^d$
only uses that $\d$ is assumed to be a flat coderivation, {\em i.e.} we have
$\Delta\d=(\id\otimes\d)\Delta+(\d\otimes\id)\Delta$
and $\d^2=0$ quite analog to the previous proof, we have
$$(\id\otimes\Delta^\d)\Delta^\d=\biggr(\id+\I\hbar\bigr[\d\otimes\d\otimes1+\d\otimes1\otimes\d+1\otimes\d\otimes\d\bigr]\biggr)(\id\otimes\Delta)\Delta$$
where the signs of \ref{ASS} are hidden in the Koszul sign conventions.\end{proof}

Consider the category $\mathrm{DGCA}$ with objects $(C,\Delta_C,\d_C)$ where $C$ is a graded coalgebra with coproduct $\Delta_C:C\rightarrow{C}\otimes{C}$ and flat coderivation $\d_C:C\rightarrow{C}$, and arrows in $\mathrm{DC}$ are by definition graded structure compatible $\mathbb{C}$-linear morphisms, {\em i.e.} for any
$\varphi:(C,\Delta_C,\d_C)[[\hbar]]\rightarrow(D,\Delta_D,\d_D)[[\hbar]]$ we assume 
$\varphi\d_C=\d_D\varphi$ and $(\varphi\otimes\varphi)\Delta_C=\Delta_D\varphi$.
\begin{proposition}
The change of objects $\mathcal{D}(C,\Delta_C,\d_C)=(C,\Delta_C^{\d_C},{d_C})[[\hbar]]$ with help of \ref{Functorialcodeformation} defines a faithful deformation functor ${\mathrm{DC}}\hookrightarrow{\mathrm{DC}}[[\hbar]]$.
\end{proposition}

\vfill{\begin{center}
\textcircled{c} Copyright by Johannes L\"offler, 2014, All Rights Reserved
\end{center}}

\end{document}